\documentclass{article}

\usepackage{amsmath}
\usepackage{amsthm}
\usepackage{amsxtra}
\usepackage{amsfonts,amssymb}
\newtheorem{Th}{Theorem}
\newtheorem{Prop}{Proposition}
\newtheorem{Lm}[Prop]{Lemma}

\theoremstyle{definition}

\newtheorem{Rem}{Remark}

\date{}
\title{Schur–Weyl duality for the unitary groups of ${\rm II}_1$-factors}

\author{Nessonov ~N.~I. \footnote{This research was supported in part by the grant “Network of Mathematical
Research 2013--2015}}

\begin{document}
\maketitle
\begin{abstract}
We obtain the analogue of  Schur-Weyl duality for the unitary group of an arbitrary ${\rm II}_1$-factor
\end{abstract}
\section{Preliminaries.}
Let $\mathcal{M}$ be  a separable ${\rm II}_1$-factor , let $U\left(\mathcal{M} \right)$ be its unitary group and let ${\rm tr}$ be a unique normalized normal trace on $M$. Denote by $\mathcal{M}^\prime$ commutant of $\mathcal{M}$. Assume that $\mathcal{M}$ acts on $L^2\left(\mathcal{M}, {\rm tr}\right)$ by  left multiplication: $\mathfrak{L}(a)\eta=a\eta$, where $a\in \mathcal{M}$, $\eta\in L^2\left(\mathcal{M}, {\rm tr}\right)$. Then $\mathcal{M}^\prime$ coincides with the set of the operators that act on $L^2\left(\mathcal{M}, {\rm tr}\right)$ by right multiplication: $\mathfrak{R}(a)\eta=\eta a,\;$ where  $\eta\in L^2\left(\mathcal{M}, {\rm tr}\right)$, $a\in\mathcal{M} $.   Let $\mathfrak{S}_p$  be the symmetric group of the $n$ symbols $1$, $2$, $\ldots$, $p$. Take $u\in U(\mathcal{M})$  and define  the operators $\mathfrak{L}^{\otimes p}(u)$  and $\mathfrak{R}^{\otimes p}(u)$ on $L^2\left(\mathcal{M}, {\rm tr}\right)^{\otimes p}$ as
 follows
\begin{eqnarray*}
&\mathfrak{L}^{\otimes p}(u) \left(x_1\otimes x_2\otimes\cdots\otimes x_p \right)=ux_1\otimes ux_2\otimes\cdots\otimes ux_p,&\\
&\mathfrak{R}^{\otimes p}(u) \left(x_1\otimes x_2\otimes\cdots\otimes x_p \right)=x_1u^*\otimes x_2u^*\otimes\cdots\otimes x_pu^*,\\
&\text{ where } \;\;x_1,x_2,\ldots,x_p\in L^2\left(\mathcal{M}, {\rm tr}\right).
\end{eqnarray*}
Obviously the operators $\mathfrak{L}^{\otimes p}(u)$ and $\mathfrak{R}^{\otimes p}(u)$, where $u\in U(\mathcal{M})$, form the unitary representations of the group $U(\mathcal{M})$.
Also, we define the representation $\mathcal{P}_p$ of $\mathfrak{S}_p$ that acts on $L^2\left(\mathcal{M}, {\rm tr}\right)^{\otimes p}$ by
\begin{eqnarray}\label{Unitary_permutation}
\mathcal{P}_p(s)\left(x_1\otimes x_2\otimes\cdots\otimes x_p \right)=x_{s^{-1}(1)}\otimes x_{s^{-1}(2)}\otimes\cdots\otimes x_{s^{-1}(p)}, s\in\mathfrak{S}_p.
\end{eqnarray}

Denote by ${\rm Aut}\, \mathcal{M}$ the automorphism group of factor
$\mathcal{M}$. Let $\theta_p^s$ be the automorphism of factor $\mathcal{M}^{\otimes p}$ that acts as follows
\begin{eqnarray}\label{aut_permutation}
\theta_p^s(a)=\mathcal{P}_p(s)a\mathcal{P}_p(s^{-1}), \text{ where } s\in\mathfrak{S}_p, a\in \mathcal{M}^{\otimes p}\cup\mathcal{M}^{\prime\otimes p}.
\end{eqnarray}

Let  $\mathcal{A}$ be  the set of the operators  on Hilbert space $H$, let $\mathcal{N}_\mathcal{A}$ be the smallest von Neumann algebra containing  $\mathcal{A}$, and let $\mathcal{A}^\prime$ be a commutant of $\mathcal{A}$. By von Neumann's bicommutant theorem
$
\mathcal{N}_\mathcal{A}=\left\{ \mathcal{A}^\prime \right\}^\prime=\mathcal{A}^{\prime\prime}$.

 Set $\left(\mathcal{M}^{\otimes p} \right)^{\mathfrak{S}_p}$ = $\left\{ a\in\mathcal{M}^{\otimes p}: \theta_p^s(a)=a \text{ for all } s\in\mathfrak{S}_p \right\}$.

 The irreducible representations of $\mathfrak{S}_p$ are indexed by the partitions\footnote{A partition $\lambda=\left(\lambda_1,\lambda_2,\ldots  \right)$ is a weakly decreasing sequence of non-negative integers $\lambda_j$, such that $\sum\lambda_j=p$. As usual, we write $\lambda$ $\vdash p$.} of $p$. Let $\lambda$ be a partition of $p$, and let $\chi^\lambda$ be the character of the corresponding irreducible representation $T^\lambda$. If ${\rm dim}\,\lambda$ is the dimension of $T^\lambda$, then operator $P_p^\lambda=\frac{{\rm dim}\,\lambda}{p\,!}\sum\limits_{s\in\mathfrak{S}}\chi^\lambda(s)\mathcal{P}_p(s)$ is the orthogonal projection on  $L^2\left(\mathcal{M}, {\rm tr}\right)^{\otimes p}$. Denote by $\Upsilon_p$ the set of all partitions of $p$.
The following statement is an analogue of the Schur-Weil duality.

\begin{Th}\label{main_theorem}
Fix the nonnegative  integer numbers $p$ and $q$. Let $\lambda$ and $\mu$ be the partitions from $\Upsilon_p$ and $\Upsilon_q$, respectively, and let $\Pi_{\lambda\mu}$ be the restriction of representation $\mathfrak{L}^{\otimes p}\otimes\mathfrak{R}^{\otimes q}$ to the subspace $\mathcal{H}_{\lambda\mu}=P_p^\lambda\otimes P_q^\mu \left(L^2\left(\mathcal{M}, {\rm tr}\right)^{\otimes p}\otimes L^2\left(\mathcal{M}, {\rm tr}\right)^{\otimes q}\right)$. The following properties are true.
\begin{itemize}
  \item {\rm (1)} $\left\{\mathfrak{L}^{\otimes p}\otimes\mathfrak{R}^{\otimes q}\left( U(\mathcal{M})\right)\right\}^{\prime\prime}=\left(\mathcal{M}^{\otimes p}\right)^{\mathfrak{S}_p}\otimes\left(\mathcal{M}^{\prime\otimes q} \right)^{\mathfrak{S}_q}$. In particular, the algebra $\left(\mathcal{M}^{\otimes p}\right)^{\mathfrak{S}_p}\otimes\left(\mathcal{M}^{\prime\otimes q} \right)^{\mathfrak{S}_q}$ is the finite factor.
  \item {\rm (2)} For any $\lambda$ and $\mu$ the representation  $\Pi_{\lambda\mu}$ is quasi-equivalent to $\mathfrak{L}^{\otimes p}\otimes\mathfrak{R}^{\otimes q}$.
  \item {\rm (3)} Let $\gamma\vdash p$ and $\delta\vdash q$. The representations $\Pi_{\lambda\mu}$ and $\Pi_{\gamma\delta}$ are unitary equivalent if and only if ${\rm dim}\,\lambda\cdot {\rm dim}\,\mu={\rm dim}\,\gamma\cdot{\rm dim}\,\delta$.
\end{itemize}
\end{Th}
\section{The proof of property (1)} In this section we give three  auxiliary lemmas  and the proof of property (1) in theorem \ref{main_theorem}.
\begin{Lm}\label{approximation}
Let $A$ be a self-adjoint operator from ${\rm II}_1$-factor $\mathcal{M}$. Then for any number $\epsilon>0$, there exist a hyperfinite ${\rm II}_1$-subfactor $\mathcal{R}_0\subset\mathcal{M}$ and self-adjoint ope\-rator $A_\epsilon\in\mathcal{R}_0$ such that $\left\|A-A_\epsilon \right\|<\epsilon$.
Here $\|\;\;\|$ is the ordinary operator norm.
\end{Lm}
\begin{proof}
Let $A=\int\limits_a^b t\,{\rm d}E_t$ be the spectral decomposition of $A$. Fix the increasing finite sequence of the real numbers $a=a_1 <a_2 < \ldots<a_m>b$ such that $\left| a_i-a_{i+1} \right|<\epsilon$.
Hence, choosing $t_i\in\left[a_i\,a_{i+1} \right)$, we have
\begin{eqnarray}\label{epsilon_approxim}
\left\| A-\sum\limits_{i=1}^m t_i\,E_{[a_i\, a_{i+1})}\right\|<\epsilon.
\end{eqnarray}
It is obvious that $\mathcal{M}$ contains the sequence of the pairwise commuting ${\rm I}_2$-sub\-factors $M_i$, where $i\in \mathbb{N}$. Notice that  exist the pairwise orthogonal projections $F_i$ from the hyperfinite  ${\rm II}_1$-factor $\left(\bigcup\limits_i M_i\right)^{\prime\prime}$ and unitary $u\in\mathcal{M}$ such that
\begin{eqnarray*}
E_{[a_i\, a_{i+1})} = uF_iu^* \text{ for } i=1,2,\ldots,m-1.
\end{eqnarray*}
It follows from (\ref{epsilon_approxim}) that $\mathcal{R}_0=u\left(\bigcup\limits_i M_i\right)^{\prime\prime}u^*$ and $A_\epsilon= \sum\limits_{i=1}^m t_i\,E_{[a_i\, a_{i+1})}\in \mathcal{R}_0$ satisfy the conditions as in the lemma.
\end{proof}
Consider the operators $\mathfrak{l}(a)$ and $\mathfrak{r}(a)$, where $a\in \mathcal{M}$, acting in Hilbert space $L^2\left(\mathcal{M,{\rm tr}} \right)$ by
\begin{eqnarray*}
\mathfrak{l}(a\eta)=a\eta, \mathfrak{r}(a)=\eta a,\;\;\eta\in L^2\left(\mathcal{M,{\rm tr}} \right).
\end{eqnarray*}
Let us denote by $^k\!a$ the operator $\underbrace{{\rm
I}\otimes\cdots\otimes{\rm I}\otimes\stackrel{k}{A}\otimes{\rm
I}\otimes\cdots)}_{p+q}\in \mathcal{M}^{\otimes p}\otimes\mathcal{M}^{\prime\otimes q}$, where $A=\left\{\begin{array}{rl} \mathfrak{l}(a)& \text{ if } k\leq p\\
\mathfrak{r}(a)& \text{ if } p<k\leq p+q\end{array}\right.$, $a\in\mathcal{M}$.
\begin{Lm}\label{pqT}
Operator $\;\;\; ^{pq}\!T(a)=\sum\limits_{k=1}^p {^k\!\mathfrak{l}(a)}-\sum\limits_{k=p+1}^{p+q}{^k\!\mathfrak{r}(a)}\;\;\;$ lies in algebra
$\left\{\mathfrak{L}^{\otimes p}\otimes\mathfrak{R}^{\otimes q}\left( U(\mathcal{M})\right)\right\}^{\prime\prime}$ for all $a\in\mathcal{M}$.
\end{Lm}
\begin{proof}
Fix any self-adjoint operator $A\in \mathcal{M}$ and consider the one-parameter unitary group $u_t=e^{itA}\in U(\mathcal{M})$, $t\in\mathbb{ R}$. It is clear that $$\frac{{\rm d}}{{\rm d}\,t}\mathfrak{L}^{\otimes p}\otimes\mathfrak{R}^{\otimes q}\left(u_t \right)\Big|_{t=0}= \; ^{pq}\!T(A)\in
\left\{\mathfrak{L}^{\otimes p}\otimes\mathfrak{R}^{\otimes q}\left( U(\mathcal{M})\right)\right\}^{\prime\prime}.$$
It follows from this that $^{pq}\!T(A)+i\,\,^{pq}T(B)=\;^{pq}T(A+iB)\in \left\{\mathfrak{L}^{\otimes p}\otimes\mathfrak{R}^{\otimes q}\left( U(\mathcal{M})\right)\right\}^{\prime\prime}$ for any self-adjoint operator $B\in\mathcal{M}$.
\end{proof}
\begin{Lm}\label{+-}
Algebra $\left\{\mathfrak{L}^{\otimes p}\otimes\mathfrak{R}^{\otimes q}\left( U(\mathcal{M})\right)\right\}^{\prime\prime}$ contains the operators $ \; ^{p}T^+(a)$ $=\sum\limits_{k=1}^p {^k\!\mathfrak{l}(a)}$ and $\; ^{q}T^-(a)=\sum\limits_{k=p+1}^{p+q} {^k\mathfrak{r}(a)}$ for all $a\in\mathcal{M}$.
\end{Lm}
\begin{proof}
At first we will prove that  $\; ^{p}T^+(a)$ lies in $\left\{\mathfrak{L}^{\otimes p}\otimes\mathfrak{R}^{\otimes q}\left( U(\mathcal{M})\right)\right\}^{\prime\prime}$.

Take self-adjoint $a\in\mathcal{M}$ and fix number $\epsilon>0$.
Using lemma \ref{approximation},  we find  the hyperfinite ${\rm II}_1$-factor $\mathcal{R}_0$ and $a_\epsilon\in\mathcal{R}_0$ such that
\begin{eqnarray}\label{estimate_epsilon}
\left\| a-a_\epsilon \right\|<\epsilon.
\end{eqnarray}
Let $M_i$, $i\in\mathbb{N}$ be the sequence of pairwise commuting ${\rm I}_2$-subfactors from $\mathcal{R}_0$ such that $\left\{ \bigcup\limits_{j\in\mathbb{N}} M_j\right\}^{\prime\prime}=\mathcal{R}_0$, and $\mathcal{N}_0$ be the relative commutant of $\mathcal{R}_0$ in $\mathcal{M}$: $\mathcal{N}_0$ $=$ $\mathcal{R}_0^\prime\cap\mathcal{M}$.
There exists the unique normal conditional expectation $\mathcal{E}$ of $\mathcal{M}$ onto $\mathcal{N}_0$ satisfying the next conditions
\begin{itemize}
  \item {\bf a)} ${\rm tr}\left(a \right)={\rm tr}\left(\mathcal{E}(a) \right)$ for all $a\in\mathcal{M}$;
  \item {\bf b)} $\mathcal{E}\left(xay \right)=x\mathcal{E}\left(a\right)y$  for all $a\in\mathcal{M}$ and $x,y\in\mathcal{N}_0$;
  \item {\bf c)} $\mathcal{E}(a)={\rm tr}(a)$  for all $a\in\mathcal{R}_0$.
\end{itemize}
Denote by $U_k\left(2^{l-k} \right)$ the unitary subgroup of the ${\rm I}_{2^{l-k}}$-factor $\left\{ \bigcup\limits_{j=k+1}^lM_j \right\}^{\prime\prime}$. Let ${\rm d}\,u$ be Haar measure on $U_k\left(2^{l-k} \right)$.

Since, by lemma \ref{pqT},
\begin{eqnarray*}
\;\;\; ^{pq}\!T(au^*)\,\cdot\, ^{pq}\!T(u)\;\; \text{ lies in }\;
\left\{\mathfrak{L}^{\otimes p}\otimes\mathfrak{R}^{\otimes q}\left( U(\mathcal{M})\right)\right\}^{\prime\prime} \;\text{ for all }\; a,u\in\mathcal{M},
\end{eqnarray*}
 to prove the theorem, it suffices to show that
\begin{eqnarray}\label{limit}
\lim\limits_{n\to\infty} \;\;\int\limits_{U_0\left( 2^{n}\right)}\;^{pq}\!T(au^*)\,\cdot\, ^{pq}\!T(u)\,{\rm d}\,u=
\; ^{p}T^+(a)-\; ^{q}T^-\left(\mathcal{E}(a) \right), a\in \mathcal{M}
\end{eqnarray}
with respect to the strong operator topology.

Indeed, then, by property {\bf c)},  the operator $ \; ^{p}T^+(a_\epsilon)-q{\rm tr}(a_\epsilon) {\rm I} $, where ${\rm I}$ is the identity operator from $\mathcal{M}^{\otimes p}\otimes\mathcal{M}^{\prime\otimes q}$, lies in $\left\{\mathfrak{L}^{\otimes p}\otimes\mathfrak{R}^{\otimes q}\left( U(\mathcal{M})\right)\right\}^{\prime\prime}$. Hence, using (\ref{estimate_epsilon}), we obtain
\begin{eqnarray*}
\; ^{p}T^+(a)\in\left\{\mathfrak{L}^{\otimes p}\otimes\mathfrak{R}^{\otimes q}\left( U(\mathcal{M})\right)\right\}^{\prime\prime}.
\end{eqnarray*}
To calculate of the left side in (\ref{limit}) we notice that
\begin{eqnarray}\label{product_TT}
\;^{pq}\!T(au^*)\,\cdot\, ^{pq}\!T(u)=\sum\limits_{k=1}^p\; ^k\mathfrak{l}(a)+\sum\limits_{k=p+1}^{p+q}
\; ^k\mathfrak{r}\left(uau^* \right)+ \Sigma(a,u), \;\text{ where }
\end{eqnarray}
\begin{eqnarray}\label{sigma}
\begin{split}
&\Sigma(a,u)=\sum\limits_{\{k,j=1\}\&\{k\neq j\}}^p \, ^k\!\mathfrak{l}(a)\cdot\, ^k\!\mathfrak{l}(u^*)\cdot\, ^j\mathfrak{l}(u)+\sum\limits_{\{k,j=p+1\}\&\{k\neq j\}}^{p+q} \, ^k\!\mathfrak{r}(u^*)\cdot\, ^j\mathfrak{r}(u)\cdot\, ^k\mathfrak{r}(a)\;\;\;\;\;\;\\
&-\sum\limits_{k=1}^p\sum\limits_{j=p+1}^{p+q}\; ^k\mathfrak{l}(a)\cdot\; ^k\!\mathfrak{l}\left( u^* \right)\cdot\; ^j\mathfrak{r}(u)-\sum\limits_{k=1}^p\sum\limits_{j=p+1}^{p+q}\;
  ^k\mathfrak{l}(u)\cdot\; ^j\mathfrak{r}\left( u^* \right)\cdot\;^j\mathfrak{r}(a).
 \end{split}
\end{eqnarray}
Let us first prove that
\begin{eqnarray}\label{limit_exp}
\lim\limits_{n\to \infty}\int\limits_{U_0\left(2^n \right)} \; ^k\mathfrak{r}\left(uau^* \right)\,{\rm}d
\,u\;= \; ^k\mathfrak{r}\left(\mathcal{E}(a) \right) \;\text{ for all } a\in \mathcal{M}
\end{eqnarray}
with respect to the strong operator topology.

For this purpose we notice that the map
\begin{eqnarray*}
a\ni L^2\left(\mathcal{M,{\rm tr}} \right)\stackrel{\mathcal{E}_n}{\mapsto}\int\limits_{U_0\left(2^n \right)}uau^*\,{\rm}d
\,u\;\in L^2\left(\mathcal{M,{\rm tr}} \right)
\end{eqnarray*}
 is the orthogonal projection. Since $\mathcal{E}_n\geq \mathcal{E}_{n+1}$, then
 \begin{eqnarray*}
 \lim\limits_{n\to\infty}\mathcal{E}_n(a)=\mathcal{E}(a) \;\text{ for all } a \in \mathcal{M}
 \end{eqnarray*}
with respect to the norm on   $L^2\left(\mathcal{M,{\rm tr}} \right)$. Hence, applying the inequality
$\left\|\mathcal{E}_n(a)\right\|$ $\leq\|a\|$, we obtain $\lim\limits_{n\to\infty}\left\|\mathcal{E}_n(a)\eta-\mathcal{E}(a)\eta\right\|_{L^2}=0$ for all $\eta\in
 L^2\left(\mathcal{M,{\rm tr}} \right)$. This gives (\ref{limit_exp}).

 To estimate of $\Sigma(a,u)$ fix the matrix unit $\left\{ \mathfrak{e}_{pq}: 1\leq p,q\leq 2^n \right\}$ of
 the ${\rm I}_{2^n}$-factor $\left\{ \bigcup\limits_{j=1}^n M_j \right\}^{\prime\prime}$. We recall that
the operators $e_{pq}$ satisfy the relations
\begin{eqnarray*}
\mathfrak{e}_{pq}^*=\mathfrak{e}_{qp}, \mathfrak{e}_{pq}\mathfrak{e}_{st}=\delta_{qs}\mathfrak{e}_{pt},  \;\;1\leq p,q,s,t\leq 2^n.
\end{eqnarray*}
Denote by $\left\{ a_{pq}\right\}_{p,q=1}^{2^n} \subset \mathbb{C}$ the corresponding matrix elements of the operator $a\in \left\{ \bigcup\limits_{j=1}^n M_j \right\}^{\prime\prime}$: $a=\sum\limits_{p,q=1}^{2^n}a_{pq}\mathfrak{e}_{pq}$. If $k\neq j$, then, applying Peter-Weyl theorem, we obtain
\begin{eqnarray*}
\begin{split}
^{kj}T_\mathfrak{l}=\int\limits_{U_0\left(2^n\right)}\,^k\mathfrak{l}\left(u^* \right)\cdot
\,^j\mathfrak{l}\left(u\right)\;{\rm d}\,u=2^{-n}\sum\limits_{p,q=1}^{2^n}
\,^k\mathfrak{l}\left(\mathfrak{e}_{pq} \right)\cdot\,^j\mathfrak{l}\left(\mathfrak{e}_{qp}\right),\\
^{kj}T_\mathfrak{r}=\int\limits_{U_0\left(2^n\right)}\,^k\mathfrak{r}\left(u^* \right)\cdot
\,^j\mathfrak{r}\left(u\right)\;{\rm d}\,u=2^{-n}\sum\limits_{p,q=1}^{2^n}
\,^k\mathfrak{r}\left(\mathfrak{e}_{pq} \right)\cdot\,^j\mathfrak{r}\left(\mathfrak{e}_{qp}\right),\\
\,^{kj}P=\int\limits_{U_0\left(2^n\right)}\,^k\mathfrak{l}\left(u^* \right)\cdot
\,^j\mathfrak{r}\left(u\right)\;{\rm d}\,u=2^{-n}\sum\limits_{p,q=1}^{2^n}
\,^k\mathfrak{l}\left(\mathfrak{e}_{pq} \right)\cdot\,^j\mathfrak{r}\left(\mathfrak{e}_{qp}\right).
\end{split}
\end{eqnarray*}
A trivial verification shows that
\begin{eqnarray*}
&\left(^{kj}T_\mathfrak{l} \right)^*=\;^{kj}T_\mathfrak{l}, \left(^{kj}T_\mathfrak{r} \right)^*=\, ^{kj}T_\mathfrak{r},\; ^{kj}T_\mathfrak{l}^2=\,^{kj}T_\mathfrak{r}^2=2^{-2n}{\rm I},\\
&\left(\,^{kj}\!P \right)^*=\,^{kj}\!P^2=2^{-n}\cdot\,^{kj}\!P.
\end{eqnarray*}
Hence, using (\ref{sigma}), we have
\begin{eqnarray*}
\lim\limits_{n\to\infty}\int\limits_{U_0\left(2^n \right)}\Sigma(a,u)\;{\rm d}\,u=0
\end{eqnarray*}
with respect to the operator norm. We thus get (\ref{limit}).

The proof above works  for the operator $\; ^{q}T^-(a)$. But we must examine
$\;^{pq}\!T(u^*a)\,\cdot\, ^{pq}\!T(u)$ instead $\;^{pq}\!T(au^*)\,\cdot\, ^{pq}\!T(u)$ (see (\ref{product_TT})).
\end{proof}
\paragraph*{The proof of Theorem \ref{main_theorem}(1).}
By lemma \ref{+-}, it suffices to show that
\begin{eqnarray}\label{UT}
\left\{ \; ^{p}T^+(a), a\in\mathcal{M} \right\}^{\prime\prime}=\left(\mathcal{M}^{\otimes p} \right)^{\mathfrak{S}_p}.
\end{eqnarray}
Fix the orthonormal bases $\left\{ b_j \right\}_{j=0}^\infty$ in $L^2\left(\mathcal{M},{\rm tr} \right)$ such that $b_j\in\mathcal{M}$ and $b_0={\rm I}$. Let $\mathbf{j}=\left(j_1,j_2,\ldots,j_p \right)$ be the ordered collection of the indexes,  and let $\mathbf{b}_\mathbf{j}$ = $ b_{j_1}\otimes b_{j_2}\otimes\ldots\otimes b_{j_p}$ be the corresponding element in $L^2\left(\mathcal{M}^{\otimes p}, {\rm tr}^{\otimes p} \right)\cap\mathcal{M}^{\otimes p}$. We call two collections $\mathbf{i}=\left(i_1,i_2,\ldots,i_p \right)$ and $\mathbf{j}=\left(j_1,j_2,\ldots,j_p \right)$ are equivalent if there exists $s\in \mathfrak{S}$ such that
$\left(i_1,i_2,\ldots,i_p \right)$ = $\left(j_{s(1)},j_{s(2)},\ldots,j_{s(p)} \right)$. Denote by $\overline{\mathbf{i}}$ the equivalence class containing $\mathbf{i}$. Set $s(\mathbf{j})=\left(j_{s(1)},j_{s(2)},\ldots,j_{s(p)} \right)$, $s\in\mathfrak{S}_p$. It is clear that the elements $\mathbf{b}_{\overline{\mathbf{j}}}=\sum\limits_{s\in\mathfrak{S}_p}\mathbf{b}_{s(\mathbf{j})}
\in \left(\mathcal{M}^{\otimes p} \right)^{\mathfrak{S}_p}$
form the orthogonal bases in $L^2\left(\left(\mathcal{M}^{\otimes p} \right)^{\mathfrak{S}_p},{\rm tr}^{\otimes p} \right)$. So to prove (\ref{UT}) , it suffices to show that
\begin{eqnarray}\label{BU}
\mathbf{b}_{\overline{\mathbf{j}}}\in \left\{ \; ^{p}T^+(a), a\in\mathcal{M} \right\}^{\prime\prime}{\rm I}.
\end{eqnarray}

 Denote by $L_m$  the span of the elements $\mathbf{b}_{\overline{\mathbf{j}}}$ such that
$\left| \left\{ k: j_k>0 \right\} \right|=m$. It is obvious that $L_m\subset\left(\mathcal{M}^{\otimes p} \right)^{\mathfrak{S}_p}$. In particular, $H_0=\mathbb{C}\,{\rm I}$. It follows easily that $L_1$ is the  set $\left\{ \; ^{p}T^+(a)\,{\rm I}, a\in\mathcal{M} \right\}$. A trivial verification shows that $L_m$ are pairwise orthogonal and the closure of $\bigoplus\limits_{m=0}^p L_m$ with respect to the $L^2$-norm coincides with $L^2\left(\left(\mathcal{M}^{\otimes p} \right)^{\mathfrak{S}_p},{\rm tr}^{\otimes p} \right)$.
Thus, if can we prove that
\begin{eqnarray}\label{LmT}
L_m\subset \left\{ \; ^{p}T^+(a), a\in\mathcal{M} \right\}^{\prime\prime}{\rm I} \text{ for all } m=1,2,\ldots,p,
\end{eqnarray}
then we obtain (\ref{BU}).

Let us prove this, by induction on $m$.

If $m=1$ then $L_1\subset \left\{ \; ^{p}T^+(a), a\in\mathcal{M} \right\}^{\prime\prime}{\rm I}$, by the definition of $\; ^{p}T^+(a)$ (see lemma \ref{+-}). Assuming (\ref{LmT}) to hold for $m=1,2,\ldots,k$, we will prove that
\begin{eqnarray}
L_{k+1}\subset \left\{ \; ^{p}T^+(a), a\in\mathcal{M} \right\}^{\prime\prime}{\rm I}.
\end{eqnarray}
Indeed, if $\mathbf{b}_{\overline{\mathbf{j}}}$ lies in $L_k$ then without loss of generality we can assume that
\begin{eqnarray*}
{\mathbf{j}}=\underbrace{\left(\overbrace{j_1,j_2,\ldots,j_k}^k,0,\ldots,0 \right)}_p,\;\text{ where } j_i\neq0 \text{ for all } i\in\left\{ 1,2,\ldots,k \right\}.
\end{eqnarray*}
If $l\neq 0$ then
$
^{p}T^+\left(b_l \right)\mathbf{b}_{\overline{\mathbf{j}}}=b_l^{(k)}+ \mathbf{b}_{\overline{\mathbf{i}}}
$, where $b_l^{(k)}\in \bigoplus\limits_{m=0}^k L_m\subset\left\{ \; ^{p}T^+(a), a\in\mathcal{M} \right\}^{\prime\prime}{\rm I}$ and
$
{\mathbf{i}}=\underbrace{\left(\overbrace{j_1,j_2,\ldots,j_k,l}^{k+1},0,\ldots,0 \right)}_p.
$
 Therefore, $\mathbf{b}_{\overline{\mathbf{i}}}$ lies in $\left\{ \; ^{p}T^+(a), a\in\mathcal{M} \right\}^{\prime\prime}{\rm I}$. This proves (\ref{LmT}), (\ref{BU}) and (\ref{UT}).\qed

\section{The proof of the properties (2) and (3)}
Let ${\rm Aut}\,\mathcal{N}$ be the group of all automorphisms of von Neumann algebra $\mathcal{N}$.
We recall that automorphism $\theta$ of factor $\mathcal{F}$ is inner if there exists unitary $u\in\mathcal{F}$ such that $\theta(a)= uau^*={\rm Ad}\,u(a)$. Let us denote by ${\rm Int}\,\mathcal{F}$ the set of all inner automorphisms of factor $\mathcal{F}$. An automorphism $\theta\in{\rm Aut}\,\mathcal{F}$ is called outer if $\theta\notin{\rm Int}\,\mathcal{F}$.

Consider ${\rm II}_1$-factor $\mathcal{F} =\mathcal{M}^{\otimes p}\otimes \left(\mathcal{M}^\prime \right)^{\otimes q}$. We emphasize that $\mathcal{F}$ is generated by the operators $A=\mathfrak{l}(a_1)\ldots\otimes\mathfrak{l}(a_p)\otimes\mathfrak{r}(a_{p+1}
\otimes\ldots\otimes\mathfrak{r}(a_{p+q}))$ $\left(a_j\in\mathcal{M} \right)$, $\left( 1\leq j\leq p+q\right)$ which act in $L^2\left(\mathcal{M}^{\otimes(p+q)},{\rm tr}^{\otimes(p+q)} \right)$ as follows
\begin{eqnarray*}
A\left(\eta_1\otimes\ldots\eta_p\otimes\eta_{p+1}\otimes\ldots\eta_{p+q}\right)=
a_1\eta_1\otimes\ldots a_p\eta_p\otimes\eta_{p+1}a_{p+1}^*\otimes\ldots\eta_{p+q}a_{p+q}^*.
\end{eqnarray*}
From now on, ${\rm tr}^{\otimes(p+q)}$ denotes the unique normal normalized trace on the factor $\mathcal{F}$:
\begin{eqnarray}\label{tr_on_F}
{\rm tr}^{\otimes(p+q)}(A)=\prod\limits_{k=1}^p{\rm tr}\left(a_k \right)\prod\limits_{k=p+1}^{p+q}{\rm tr}\left(a_k^* \right).
\end{eqnarray}
If $J$ is the antilinear isometry on $L^2\left(\mathcal{M}^{\otimes(p+q)},{\rm tr}^{\otimes(p+q)} \right)$ defined by\newline $L^2\left(\mathcal{M}^{\otimes(p+q)},{\rm tr}^{\otimes(p+q)} \right)\ni X\stackrel{J}{\mapsto}X^*\in L^2\left(\mathcal{M}^{\otimes(p+q)},{\rm tr}^{\otimes(p+q)} \right)$, then
\begin{eqnarray}\label{antiisometry}
JAJ\left(\eta_1\otimes\ldots\eta_p\otimes\eta_{p+1}\otimes\ldots\eta_{p+q}\right)=
\eta_1 a_1^*\otimes\ldots \eta_pa_p^*\otimes a_{p+1}\eta_{p+1}\otimes\ldots a_{p+q}\eta_{p+q}.\;\;\;\;\;
\end{eqnarray}
Well-known that $\mathcal{F}^\prime=J\mathcal{F}J$ (see\cite{TAKES2}).

Let $\mathcal{P}_{p+q}(s)$, $\left(s\in\mathfrak{S}_{p+q} \right)$ be the unitary operator on   $L^2\left(\mathcal{M}^{\otimes(p+q)},{\rm tr}^{\otimes(p+q)} \right)$ defined by (\ref{Unitary_permutation}), and let $\mathfrak{S}_p\times\mathfrak{S}_q=\left\{ s\in\mathfrak{S}_{p+q}:s\left\{ 1,2,\ldots,p \right\}=\left\{ 1,2,\ldots,p \right\} \right\}$. Denote by $e$ the unit in the group $\mathfrak{S}_{p+q}$.
The next lemma  is obvious from the definition of factor $\mathcal{F}$.
\begin{Lm}\label{automorphism}
For each $s\in\mathfrak{S}_p\times\mathfrak{S}_q$ the map $\;\;\;\mathcal{F}\ni a\mapsto \mathcal{P}_{p+q}(s)a\mathcal{P}_{p+q}(s^{-1})\;=\;\;$ $\theta_{p+q}^s(a)$ is the automorphism of factor $\mathcal{F}$.
\end{Lm}
\begin{Lm}\label{outer_aut}
If $s$ is any non-identical element from $\mathfrak{S}_p\times\mathfrak{S}_q$ then ${\rm Ad}\,\mathcal{P}_{p+q}$ is the outer automorphism of factor $\mathcal{F}$.
\end{Lm}
\begin{proof}
On the contrary, suppose that there exists the unitary operator $U\in \mathcal{F}$ such that
\begin{eqnarray}\label{inner}
\mathcal{P}_{p+q}(s)a\mathcal{P}_{p+q}(s^{-1})\;=UaU^* \;\text{ for all } a\in\mathcal{F}.
\end{eqnarray}
Let us prove that $U=0$. For this, it suffices to show that
\begin{eqnarray}\label{orthogonality}
{\rm tr}^{\otimes(p+q)}\left(U\left(u_1\otimes u_2\otimes\ldots\otimes u_{p+q} \right) \right)=0 \;\text{ for all unitary }\; u_j\in\mathcal{M}.
\end{eqnarray}
Prove that for any natural number $N$
\begin{eqnarray}\label{inequality}
\left|{\rm tr}^{\otimes(p+q)}\left(U\left(u_1\otimes u_2\otimes\ldots\otimes u_{p+q} \right) \right)\right|\;\leq \frac{1}{N}.
\end{eqnarray}
To this purpose we find the pairwise orthogonal projections $p_j\in\mathcal{M}$, $j=1,2,\ldots,N$ with the properties
\begin{eqnarray}\label{projections}
\sum\limits_{j=1}^N p_j={\rm I},\;\; {\rm tr}\left(p_j \right)=\frac{1}{N} \;\text{ for all }
j=1,2,\ldots,N.
\end{eqnarray}
Set $\,^k\!p_j=\cdots\otimes{\rm I}\otimes\cdots\otimes{\rm I}\otimes\stackrel{k}{p_j}\otimes{\rm I}\cdots $.
Without loss generality we can assume that $s(1)=i\neq1$. Then
\begin{eqnarray*}
&\left|{\rm tr}^{\otimes(p+q)}\left(U\left(u_1\otimes u_2\otimes\ldots\otimes u_{p+q} \right) \right)\right|\\
&\stackrel{(\ref{projections})}{=}\sum\limits_{j=1}^N\left|{\rm tr}^{\otimes(p+q)}\left(\,^1\!\!p_j\,U\left(u_1\otimes u_2\otimes\ldots\otimes u_{p+q} \right) \right)\right|\\
&=\sum\limits_{j=1}^N\left|{\rm tr}^{\otimes(p+q)}\left(\,^1\!\!p_j\,U\left(u_1\otimes u_2\otimes\ldots\otimes u_{p+q} \right)\,^1\!\!p_j \right)\right|\\
&\stackrel{(\ref{inner})}{=}\sum\limits_{j=1}^N\left|{\rm tr}^{\otimes(p+q)}\left(\,^1\!\!p_j\,\cdot\,^i\!\left( u_1p_ju_1^* \right)U\left(u_1\otimes u_2\otimes\ldots\otimes u_{p+q} \right) \right)\right|\\
&\leq \sum\limits_{j=1}^N{\rm tr}^{\otimes(p+q)}\left(\,^1\!\!p_j\,\cdot\,^i\!\left( u_1p_ju_1^* \right) \right)=\sum\limits_{j=1}^N{\rm tr}\left(p_j \right){\rm tr}\left(u_1p_ju_1^* \right)\stackrel{(\ref{projections})}{=} \,\frac{1}{N}.
\end{eqnarray*}
This establishes  (\ref{inequality}) and (\ref{orthogonality}).
\end{proof}
To simplify notation, we will write $\theta_s$ instead $\theta_{p+q}^s={\rm Ad}\,\mathcal{P}_{p+q}(s)$ (see lemma \ref{automorphism}).

Now we consider the crossed product $\mathcal{F}\rtimes_\theta \left(\mathfrak{S}_p\times\mathfrak{S}_q \right)$ of the factor $\mathcal{F}$ by the finite group $\mathfrak{S}_p\times\mathfrak{S}_q $ acting via $\theta: s\in \mathfrak{S}_p\times\mathfrak{S}_q \mapsto \theta_s \in{\rm Aut}\,\mathcal{F}$.

Von Neumann algebra $\mathcal{F}\rtimes_\theta \left(\mathfrak{S}_p\times\mathfrak{S}_q \right)$ is generated in Hilbert space $l^2\left(G,\mathcal{H} \right)$, where $G=\mathfrak{S}_p\times\mathfrak{S}_q$, $\mathcal{H}=L^2\left(\mathcal{M}^{\otimes(p+q)},{\rm tr}^{\otimes(p+q)} \right)$, by the operators $\Pi_\theta(a)$, $a\in F$ and $\lambda_g$, $g\in G$, which act as follows
\begin{eqnarray}\label{cross_F}
\begin{split}
&\left(\Pi_\theta(a)\eta \right)(g)=\theta_{g^{-1}}(a)\eta(g),\, \eta\in l^2\left(G,\mathcal{H} \right),\\
&\left(\lambda_s\eta \right)(g)=\eta\left(s^{-1}g \right),\;s\in G.
\end{split}
\end{eqnarray}

\begin{Rem}\label{r1}
Let $a\in L^2\left(M^{\otimes(p+q)},{\rm tr}^{\otimes(p+q)} \right)$. Set $\xi_a(g)=\left\{\begin{array}{rl} a& \text{ if } g= e\\
0& \text{ if } g\neq e\end{array}\right.$. It is easy to check that $\xi_{\rm I}$ is the cyclic vector for $\mathcal{F}\rtimes_\theta \left(\mathfrak{S}_p\times\mathfrak{S}_q \right)$. Namely, the
set of the vectors $A\xi_{\rm I}$, $a\in \mathcal{F}\rtimes_\theta \left(\mathfrak{S}_p\times\mathfrak{S}_q \right)$ is dense in $l^2\left(G,\mathcal{H} \right)$. In addition, the functional  $\hat{\tau}$  defined on $A=\sum\limits_{s\in \mathfrak{S}_p\times\mathfrak{S}_q} \Pi_\theta(a_s)\cdot \lambda_s\in  \mathcal{F}\rtimes_\theta \left(\mathfrak{S}_p\times\mathfrak{S}_q \right)$ by
$$\hat{\tau}(A)=\left(A\xi_{\rm I}, \xi_{\rm I} \right)={\rm tr}^{\otimes(p+q)}(a_e),$$
is the faithful normal trace on $\mathcal{F}\rtimes_\theta \left(\mathfrak{S}_p\times\mathfrak{S}_q \right)$.
\end{Rem}
Denote von Neumann algebra $\mathcal{F}\rtimes_\theta \left(\mathfrak{S}_p\times\mathfrak{S}_q \right)$ by $\,^\theta\!\mathcal{F}$.
\begin{Rem}\label{r2}
The involution: $\,^\theta\!\mathcal{F}\xi_{\rm I}\ni A\xi_{\rm I}\stackrel{\hat{J}}{\mapsto} A^*\xi_{\rm I}$  extends to the antilinear isometry. It follows immediately  that $\left(\hat{J}\eta \right)(x)=\theta_{x^{-1}}\left(b_{x^{-1}}^* \right)$, $\eta\in l^2(G,\mathcal{H})$. The operators $\Pi_\theta^\prime(a)=\hat{J}\Pi_\theta(a)\hat{J}$, $a\in\mathcal{F}$ and $\lambda_s^\prime=\hat{J}\lambda_s \hat{J}$, $s\in \mathfrak{S}_p\times\mathfrak{S}_q$ act by
\begin{eqnarray}
\begin{split}
&\left(\Pi_\theta^\prime(a)\eta \right)(g)=JaJ\eta(g),\, \eta\in l^2\left(G,\mathcal{H} \right),\\
&\left(\lambda_s^\prime\eta \right)(g)=\theta_s\left(\eta\left(gs\right) \right),\;s\in G.
\end{split}
\end{eqnarray}
The equality  $\hat{J}\,^\theta\!\mathcal{F}\hat{J}=\,^\theta\!\mathcal{F}^\prime$ is true. In particular, the vector $\xi_{\rm I}$ is cyclic for $\,^\theta\!\mathcal{F}^\prime$. Set ${\hat{\tau}}^\prime(A)=\hat{\tau}\left(\hat{J}A\hat{J}\right)$, where $A\in \,^\theta\!\mathcal{F}^\prime$. Then  ${\hat{\tau}}^\prime$ is the faithful normal trace on $\,^\theta\!\mathcal{F}^\prime$.
\end{Rem}
\begin{Lm}\label{factor}
Von Neumann algebra $\,^\theta\!\mathcal{F}$ \  $\left(\,^\theta\!\mathcal{F}^\prime\right)$ is ${\rm II}_1$-factor.
\end{Lm}
\begin{proof}
It follows from remarks \ref{r1} and \ref{r2} that any operator $A\in\,^\theta\!\mathcal{F}$ has a unique decomposition $A=\sum\limits_{g\in G}\Pi_\theta(a_g)\lambda_g$. Thus, if $A$ lies in the centrum of $\,^\theta\!\mathcal{F}$ then
\begin{eqnarray*}
\Pi_\theta(a_g)\lambda_g \cdot \Pi_\theta(b) =\Pi_\theta(b)\cdot\Pi_\theta(a_g)\lambda_g \;\text{ for all }\;g\in G \;\text{ and }\; b\in \mathcal{F}.
\end{eqnarray*}
Hence, using (\ref{cross_F}), we obtain
\begin{eqnarray}\label{rel_inner}
a_g\cdot\theta_g(b) =b\cdot a_g  \;\text{ for all }\;g\in G \;\text{ and }\; b\in \mathcal{F}.
\end{eqnarray}
Therefore, we have
\begin{eqnarray*}
a_g^*a_g\theta_g(b) =a_g^*b a_g , \;\;\theta_g(b^*)a_g^*a_g =a_g^*b^* a_g \;\text{ for all }\;g\in G \;\text{ and }\; b\in \mathcal{F}.
\end{eqnarray*}
Hence, we conclude
\begin{eqnarray*}
a_g^*a_g\in \mathcal{F}\cap\mathcal{F}^\prime=\mathbb{C}{\rm I} \;\text{ for all }\;g\in G.
\end{eqnarray*}
We thus get $a_g=z_gu_g$, where $z_g\in\mathbb{C}$, $u_g$ is the unitary operator from $\mathcal{F}$.
Assuming $g\neq e$, we obtain from (\ref{rel_inner}) $$\theta_g(b) =u_g^*\cdot b\cdot u_g$$ for all $b\in\mathcal{F}$. It follows from lemma \ref{outer_aut} that $a_g=0$ for all $g\neq e$. But, by (\ref{rel_inner}), $a_e\in\mathcal{F}\cap\mathcal{F}^\prime=\mathbb{C}{\rm I}$. Therefore, $A\in \mathbb{C}{\rm I}$.
\end{proof}
Let $P=\frac{1}{|G|}\sum\limits_{g\in G} \lambda_g$. We will identify $\eta \in\mathcal{H}=L^2\left(\mathcal{M}^{\otimes (p+q)},{\rm tr}^{\otimes (p+q)} \right)$ with the function $\tilde{\eta}$ $\in$ $l^2(\mathcal{H},G)$ defined by: $\tilde{\eta}(g)=\eta$  for all $g\in G$. Define the unitary operator $U_g$ on $\mathcal{H}$ by
\begin{eqnarray*}
U_g\eta = \theta_g(\eta), \;\;\eta \in \mathcal{M}^{\otimes (p+q)},
\end{eqnarray*}
where $\theta_g$ denote the automorphism ${\rm Ad}\,\mathcal{P}_{p+q}(g)$ (see lemma \ref{automorphism}).
It is easy to check  that
\begin{eqnarray}\label{projection_relations}
\begin{split}
&P \, l^2\left(\mathcal{H},G \right)=\mathcal{H},
\;\;\; P\cdot\,^\theta\!\mathcal{F}\,\cdot P=\left\{a\in\mathcal{F}:\theta_g(a)=a\;\text{ for all }\;g\in G  \right\}=\mathcal{F}^G,\\
&P\lambda_g P={\rm I} \;\text{ for all }\; g\in G,\;\;\; P\Pi_\theta (a)P= \frac{1}{|G|}\sum\limits_{g\in G}\theta_g(a)\in \mathcal{F}\\
&P\lambda_g^\prime P\eta =\theta_g(\eta)=U_g\eta,\;\;\; P\Pi_\theta^\prime (a)P\eta= JaJ\eta=\eta a^*, \text{ where }\;\;\eta\in\mathcal{H}.
\end{split}
\end{eqnarray}
We sum up this discussion  in the following.
\begin{Lm}\label{generated}
Von Neumann algebra $\left(\mathcal{F}^G\right)^\prime$ is generated by $\mathcal{F}^\prime$ and $\left\{  U_g \right\}_{g\in G}$.
\end{Lm}
\begin{proof}
According to remark \ref{r2}, von Neumann algebra $\,^\theta\!\mathcal{F}^\prime$ is generated by the operators  $\Pi_\theta^\prime (a),\;a\in \mathcal{F}$ and $\lambda_g^\prime,\; g\in G$. Hence, using (\ref{projection_relations}), we obtain the desired conclusion.
\end{proof}
\begin{Lm}
$\left(\mathcal{F}^G\right)^\prime\cap \mathcal{F}=\mathbb{C}{\rm I}$.
\end{Lm}
\begin{proof}
Let $A\in\left(\mathcal{F}^G\right)^\prime\cap \mathcal{F}=\left(\mathcal{F}^G\right)^\prime\cap (\mathcal{F}^\prime)^\prime$. Lemma \ref{generated} assures that $A=\sum\limits_{g\in G}a_gU_g$, where $a_g\in\mathcal{F}^\prime$ for all $g$. Therefore,
\begin{eqnarray*}
\sum\limits_{g\in G}ba_gU_g=\sum\limits_{g\in G}a_gU_gb \;\text{for all }\; b\in\mathcal{F}^\prime.
\end{eqnarray*}
Hence, applying lemma \ref{factor} and (\ref{projection_relations}), we have
\begin{eqnarray*}
ba_gU_g=a_gU_gb \;\text{ for all }\; b\in\mathcal{F}^\prime \text{ and } g\in G.
\end{eqnarray*}
This means that
\begin{eqnarray}\label{equality_inner}
ba_g=a_g\theta_g(b) \;\text{ for all }\; b\in\mathcal{F}^\prime \text{ and } g\in G.
\end{eqnarray}
Now we recall that, by remark \ref{r2},
Since the relations $J\mathcal{F}J=\mathcal{F}^\prime$ and $J\,U_g=U_g\,J$ are true, we obtain from lemma \ref{outer_aut}  that
\begin{eqnarray*}
\mathcal{F}^\prime\ni b\mapsto U_g \, b\,U_g^*=\theta_g(b)\in \mathcal{F}^\prime
\end{eqnarray*}
is the outer automorphism of the factor $\mathcal{F}^\prime$ for all $g\neq e$.
Now as in the proof of lemma \ref{factor}, the equality (\ref{equality_inner}) gives that $A\in \mathbb{C}{\rm I}$.
\end{proof}
\paragraph*{The proof of Theorem \ref{main_theorem}(2-3).} We recall that $\mathcal{F}=\mathcal{M}^{\otimes p}\otimes\left(\mathcal{M}^\prime \right)^{\otimes q}$ and $G=\mathfrak{S}_p\times\mathfrak{S}_q$. By theorem \ref{main_theorem}(1), $\left\{\mathfrak{L}^{\otimes p}\otimes\mathfrak{R}^{\otimes q}\left( U(\mathcal{M})\right)\right\}^{\prime\prime} =\mathcal{F}^G$. It follows from (\ref{projection_relations}) and the lemmas \ref{factor}, \ref{generated} that $\mathcal{F}^G$ is a factor, and the map
\begin{eqnarray*}
^\theta\!\mathcal{F}^\prime\ni a\stackrel{R_P}{\mapsto} PaP\in\left(\mathcal{F}^G \right)^\prime
\end{eqnarray*}
is an isomorphism. Therefore, the formula
\begin{eqnarray}\label{trace_on_commutant}
\tau^\prime\left(a \right)= \left(R_p^{-1}(a)\xi_{\rm I},\xi_{\rm I} \right),\;\; a\in \left(\mathcal{F}^G \right)^\prime (\text{ see remarks \ref{r1} and \ref{r2} })
\end{eqnarray}
defines the  normal, normalized trace on  $\left(\mathcal{F}^G \right)^\prime$.

Since the projection $P^{\lambda\mu}=P_p^\lambda\otimes P_q^\mu$ lies in $\left(\mathcal{F}^G \right)^\prime$,
the map
\begin{eqnarray}
\left\{\mathfrak{L}^{\otimes p}\otimes\mathfrak{R}^{\otimes q}\left( U(\mathcal{M})\right)\right\}^{\prime\prime} =\mathcal{F}^G\ni a\stackrel{\mathfrak{I}_{pq}^{\lambda\mu}}{\mapsto} P^{\lambda\mu}aP^{\lambda\mu}\in P^{\lambda\mu}\mathcal{F}^GP^{\lambda\mu}
\end{eqnarray}
is an isomorphism. In particular, $\mathfrak{I}_{pq}^{\lambda\mu}\left( \left(\mathfrak{R}^{\otimes p}(u)\right)\otimes\left(\mathfrak{L}^{\otimes q}(u)\right)\right)=\Pi_{\lambda\mu}(u)$ for all $u\in U(\mathcal{M})$ and $(\lambda,\mu)\in \Upsilon_p\times\Upsilon_q$. This proves the property ({\bf 2}) from the theorem \ref{main_theorem}.

To prove the property ({\bf 3}), we notice that the projections $P_p^\lambda\otimes P_q^\mu$ and $P_p^\gamma\otimes P_q^\delta$ are in $\left(\mathcal{F}^G \right)^\prime$. It follows from (\ref{trace_on_commutant}) that
\begin{eqnarray*}
\tau^\prime\left(P_p^\lambda\otimes P_q^\mu \right)=\frac{{\rm dim}\,\lambda\cdot{\rm dim}\,\mu}{p!q!}.
\end{eqnarray*}
Thus, assuming that ${\rm dim}\,\lambda\cdot{\rm dim}\,\mu={\rm dim}\,\gamma\cdot{\rm dim}\,\delta$, we obtain
\begin{eqnarray*}
\tau^\prime\left(P_p^\lambda\otimes P_q^\mu \right)=\tau^\prime\left(P_p^\gamma\otimes P_q^\delta \right).
\end{eqnarray*}
Since $\left(\mathcal{F}^G \right)^\prime$ is a factor, there exist the partial isometry $\mathfrak{U}\in\left(\mathcal{F}^G \right)^\prime$ such that
\begin{eqnarray*}
\mathfrak{U}\mathfrak{U}^*=P_p^\lambda\otimes P_q^\mu\;\text{ and }\; \mathfrak{U}^*\mathfrak{U}= P_p^\gamma\otimes P_q^\delta.
\end{eqnarray*}
Hence we have
 \begin{eqnarray*}
 &\Pi_{\lambda\mu}(u)=P_p^\lambda\otimes P_q^\mu\; \left( \mathfrak{L}^{\otimes p}\otimes\mathfrak{R}^{\otimes q}(u)\right) \;P_p^\lambda\otimes P_q^\mu
 =\mathfrak{U}\mathfrak{U}^*  \; \left( \mathfrak{L}^{\otimes p}\otimes\mathfrak{R}^{\otimes q}(u)\right) \;\mathfrak{U}\mathfrak{U}^*\\
 &=\mathfrak{U}\;\mathfrak{U}^*\mathfrak{U}  \; \left( \mathfrak{L}^{\otimes p}\otimes\mathfrak{R}^{\otimes q}(u)\right) \;\mathfrak{U}^*= \Pi_{\gamma\delta}(u) \;\text{\ \  for any\ \ \  }\;u\in U(\mathcal{M}).
 \end{eqnarray*}

{}
B.Verkin Institute for Low Temperature Physics and Engineering\\n.nessonov@gmail.com
\end{document}